\documentclass{article}
\usepackage{graphicx} 
\usepackage{amsthm, amsmath, amssymb}
\usepackage{titling}
\usepackage{xcolor}
\usepackage{enumitem}
\usepackage{hyperref}

\newtheorem{theorem}{Theorem}
\newtheorem*{maintheorem}{Main Theorem}

\newtheorem{lemma}{Lemma}
\newtheorem{remark}{Remark}

\newtheorem{corollary}{Corollary}

\newtheorem*{cor}{Corollary}
\newtheorem{example}{Example}

\newcommand{\mc}[1]{\mathcal{#1}}
\newcommand{\mrm}[1]{\mathrm{#1}}
\newcommand{\tit}[1]{\textit{#1}}
\newcommand*\samethanks[1][\value{footnote}]{\footnotemark[#1]}

\title{A Completion Result\\ for Partial Affine and Inversive Spaces}
\author{Cassie Grace \thanks{School of Mathematics and Statistics, University of Canterbury, Private Bag 4800, 8140 Christchurch, New Zealand.}
\and Klaus Metsch \thanks{Mathematisches Institut, Justus-Liebig-Universit\"{a}t Gie{\ss}en, Arndtstra{\ss}e 2, 35392 Gie{\ss}en, Germany.} \and Geertrui Van de Voorde \samethanks[1] \thanks{Corresponding author} }

\date{}

\begin{document}

\maketitle

\begin{abstract}
    A {\em partial affine plane} of order $n$ is a point-line incidence structure with $n^2$ points and $n$ points on each line, such that every two lines meet in at most one point.
    In this paper, we show that a partial affine plane of order $n$, $n\geq 19$, in which parallelism is an equivalence relation, containing more than $n^2-\sqrt{n}$ lines, can be completed to an affine plane. This bound is tight if $n$ is a square of a prime power.
    We also show that a partial affine plane of order $n$, $n\geq 49$, in which parallelism is an equivalence relation and there is no point not lying on any line, containing more than $n^2-f$ lines, where $f(f+1)=2n$, can be completed to an affine plane. These results improve on the $40$-year old bound of \cite{dow1986completion}. 
    
    Furthermore, we derive a higher-dimensional result about the completion of $2$-$(n^d,n,1)$-designs, as well as for partial {\em inversive spaces}. 
    In particular, we show that a partial $3$-$(n^2+1,n+1,1)$-design for which in every derived structure, parallelism is an equivalence relation, and there are at least $n^2+n-\sqrt{n}$ lines, can be completed to an inversive plane.
\end{abstract}

\bigskip


\section{Introduction}
One of the central and natural problems in finite geometry and design theory asks under which conditions a {\em partial} structure (projective plane, $t$-design, generalised quadrangle...) can be extended to the full structure.
This problem is well-studied for projective planes, as well as for certain partial affine planes called {\em nets} (see later), but not nearly as much is known for {\em partial affine planes}. 

An \tit{affine plane} is an incidence structure of points and lines such that through any two points there is an unique line, for any line $\ell$ and point $P\notin \ell$ there is a unique line through $P$ disjoint from $\ell$, and there exist three points which are not collinear. An affine plane of {\em order} $n$ has $n^2$ points and $n$ points on each line. 
Two lines $\ell,m$ are \tit{parallel} if $\ell=m$ or $\ell\cap m=\emptyset$. It is clear that parallelism is an equivalence relation on the set of lines in an affine plane, where equivalence classes (called {\em parallel classes}) have size $n$.

An affine plane of order $n$ is an example of a \tit{net} of order $n$: a net of order $n$ is a point-line incidence structure points with $n^2$ points and $n$ points on each line such that any two lines have at most $1$ point in common, parallelism is an equivalence relation, and each parallel class has size $n$ (note that we do not ask for any two points in a net to be contained in a common line). The problem of determining under which conditions a net may be embedded in an affine plane is well studied \cite{bruck1963finite,metsch1991improvement}, but it is important to note that these results do not immediately extend to similar partial geometries in which parallel classes are not complete.


A \tit{partial projective plane} of order $n$,  $n\geq 2$, is a point-line geometry $(\mc{S},\mc{A})$ where $\mc{S}$ is a set of $n^2+n+1$ points, and $\mc{A}$ is a set of lines (which are subsets of $\mc{S}$), such that each line has $n+1$ points and any two lines meet in at most one point. A partial projective plane with $n^2+n+1$ lines is a \tit{projective plane} of order $n$. A partial projective plane  $(\mc{S},\mc{A})$ is {\em embeddable} in a projective plane (design) if there exists a projective plane  (design) $(\mathcal{S},\mc{A'})$ such that $\mathcal{A}\subseteq \mathcal{A'},$ and we also say that this partial projective plane (design) can be {\em completed} to a full projective plane. More generally, a design $(\mc{S},\mc{A})$ is embeddable in a (full) design $(\mc{S}',\mc{A}')$ if $\mc{S}\subseteq \mc{S}'$ and for every line $L\in \mc{A}$, there is a line $L'\in \mc{A}'$ such that $L=L'\cap S$.

Partial projective planes have been extensively researched and it is known that a maximal non-embeddable partial projective plane has $n^2+1$ lines. Furthermore, such non-embeddable partial projective planes have been classified \cite{metsch1994classification}. If a partial projective plane has no disjoint lines, then this result can be improved. This problem  was first studied by Dow 
and subsequently improved by Metsch, leading to the following result.  

\begin{theorem}\label{DOW bound}
    A partial projective plane of order $n$, $n\geq 2$, with no disjoint lines and with more than either
    \begin{itemize}
    \item \cite{dow1983improved} $g_1(n)=n^2-1$, or
    \item \cite{dow1983improved}  
    $g_2(n)=n^2-2\sqrt{n+3}+6$, or
    \item \cite{klaus} $g_3(n)=n^2-n/6$, or
    \item \cite[Theorem 6.15]{klausbook} $g_4(n)=n^2-\frac{\sqrt{5}-1}{2}n+\frac{17}{\sqrt{5}}\sqrt{n}+1$
    \end{itemize}
    lines can be embedded in a projective plane of order $n$. 
\end{theorem}
It can be checked that, for integers $n\geq 2$,
\begin{itemize}\item $\min(g_1,g_2,g_3,g_4)=g_1$ if $n\leq 6$; 
\item $\min(g_1,g_2,g_3,g_4)=g_2$ if $16\leq n\leq 56$;
\item $\min(g_1,g_2,g_3,g_4)=g_3$ if $7\leq n\leq 15$ and if $57\leq n\leq 288$;
\item $\min(g_1,g_2,g_3,g_4)=g_4$ if $n\geq 289$.
\end{itemize}

A \tit{partial affine plane} (PAP) of order $n$, $n\geq 2$, is a point-line geometry $(\mc{S},\mc{A})$ where $\mc{S}$ is a set of $n^2$ points, and $\mc{A}$ is a set of lines (which are subsets of $\mc{S}$), such that each line has $n$ points and any two lines meet in at most one point.

\begin{remark} The definition of a partial affine plane and a partial projective plane only deviate in the number of total points and the number of points on a line; no partial affine plane is a partial projective plane (and vice versa) since it is not possible to find integers $m,n$ such that $m+1=n$ and $m^2+m+1=n^2$.
\end{remark}

Affine and projective planes are particular types of designs. A {\em $t$-$(v,k,\lambda)$-design} is an incidence structure with $v$ {\em points} and $k$ element subsets of points called {\em blocks}, such that every subset of $t$ points lies in exactly $\lambda$ blocks. Similarly, a {\em partial $t$-$(v,k,\lambda)$-design} is an incidence structure of $v$ points and $k$ element subsets called {\em blocks} such that every subset of $t$ points lies in at most $\lambda$ blocks. In the case $t=2$ we call the blocks {\em lines}, and in the case $t=3$ we call the blocks {\em circles}. Note that a $2$-$(n^2,n,1)$ design is an affine plane, and a partial $2$-$(n^2,n,1)$ design is a PAP. We will show in Section \ref{section2} that certain partial $2$-$(n^d,n,1)$-designs can be completed (see Lemma \ref{no_pts_lss_than_n_lines}).

A {\em $k$-group divisible design of type $t^u$} is a partial design with $tu$ points, block size $k$, and a partition of the point set into $u$ groups of size $t$ such that any two points of the same group do not lie together in a block, and any two points of different groups lie in exactly one block.

Two lines in a partial affine plane $(\mc{S},\mc{A})$ are called {\em parallel} if they coincide or have no point in common. Note that being parallel does not necessarily define an equivalence relation on the set of lines $\mc{A}$. When it does, we call the equivalence classes under the parallelism relation {\em parallel classes}. Also note that it is not necessary for the lines of a parallel class to cover all points of $\mc{S}$. Two points in a PAP $(\mc{S},\mc{A})$, (or more generally, partial design) are said to be \tit{joined} if they lie on a common line and \tit{unjoined} otherwise. Let the number of lines in a partial affine plane be $b$, then $b\leq n^2+n$ with equality if and only if $(\mc{S},\mc{A})$ is an affine plane. The \tit{valency} of a point $P$ in a PAP is the number of lines containing $P$ and is denoted $val(P)$, clearly $val(P)\leq n+1$. The number of points unjoined to $P$ is $(n+1-val(P))(n-1)$. It is not too hard to see that if parallelism is an equivalence relation on the lines and if there are at most $n+1$ parallel classes in a PAP, one can add points to the PAP and extend it to a partial projective plane. In that case, we can use the previously mentioned results about the extension of projective planes (see Lemma \ref{affine_to_proj_completion}). 

Still under the hypothesis that parallelism is an equivalence relation, but not requiring that the number of parallel classes is bounded by $n+1$, Dow showed the following:
\begin{theorem}\cite[Theorem 2.1]{dow1986completion}\label{dowaffine}
    Let $(\mc{S},\mc{A})$ be a PAP of order $n$ with $b$ lines. If $b> n^2$ and parallelism is an equivalence relation on $\mc{A}$, then $(\mc{S},\mc{A})$ can be completed to an affine plane. 
\end{theorem}


In Section \ref{section2} we improve on this result by showing the following Main Theorem (see Theorem \ref{n2-1_completable}, Theorem \ref{pap_n_minus_root_n} and Theorem \ref{pap_n_minus_root_2n}). 

\begin{maintheorem}  Let $g_i(n)$, $i=1,\ldots,4$, be as in Theorem \ref{DOW bound} and let $\min(g_1,g_2,g_3,g_4)=g$.
    Suppose that $\mc{I} = (\mc{S},\mc{A})$ is a partial affine plane of order $n$ that cannot be extended to an affine plane and suppose that parallelism is an equivalence relation. Then $|\mc{A}| \leq \mrm{max}(g(n)-1, n^2-\sqrt{n})$. Furthermore if $|A| > \mrm{max}(g(n)-1, n^2-f)$, where $f$ is the positive real number satisfying $f(f+1)=2n$, then $\mc{I}$ has a point of valency $0$ and can be embedded in a projective plane of order $n$. Moreover, if $|\mc{A}|=n^2-f$, then either $\mc{I}$ has a point of valency $0$ and can be embedded in a projective plane of order $n$, or there exists an $n$-group divisible design of type $(n-f)^{n+f+2}$ with $n^2-f$ points and lines.
\end{maintheorem}

\begin{remark}\label{rem:qs}
    It is easy to verify that $\mrm{max}(g(n)-1, n^2-\sqrt{n}) = n^2-\sqrt{n}$ when $n\geq 19$ and that $\mrm{max}(g(n)-1, n^2-f)=n^2-f$ when $n\geq 49$, and $f$ is the positive real number satisfying $f(f+1)=2n$. 
\end{remark}

We also give a construction of a partial affine plane of order $n$, where $n$ is a square, with $n^2-\sqrt{n}$ lines that cannot be extended to an affine plane (see Example \ref{ex:Baer}), which shows that the bound above is tight.

Noting that $f>\sqrt{2n}-1$, and using Remark \ref{rem:qs}, we can state the following simplified corollary.

\begin{cor}
    Let $n\geq 49$ and suppose that $\mc{I} = (\mc{S},\mc{A})$ is a partial affine plane of order $n$ in which parallelism is an equivalence relation. If  $\mc{I}$ has more than $n^2-\sqrt{n}$ lines, or $\mc{I}$ has at least $n^2-\sqrt{2n}+1$ lines and every point of $\mc{S}$ lies on at least one line of $\mc{A}$, then $\mathcal{I}$ can be extended to an affine plane of order $n$.
\end{cor}

When we do not assume that parallelism is an equivalence relation, the results are comparatively weaker. In this case, Dow proved the following theorems, which we will use in Section \ref{section3}.

\begin{theorem}\label{DOW_pap_root_n_no_parallel}\cite[Theorem 3.2, Theorem 3.3]{dow1983improved}
    Let $(\mc{S},\mc{A})$ be a PAP of order $n$ with $b=n^2+n-e$ lines. If at least one of the following holds:
    \begin{itemize}
    \item[(i)] $e<\sqrt{n}+1$ and there is a point of valency $n+1-e$, or
    \item[(ii)] $e<\sqrt{n}$ and some line contains only points of valency $n+1$,
    \end{itemize} then $(\mc{S},\mc{A})$ can be completed to an affine plane.
\end{theorem}

\begin{theorem}\label{DOW_pap_with_nothing_2_missing}\cite[Theorem 4.1]{dow1983improved}
    Let $(\mc{S},\mc{A})$ be a PAP of order $n$ with $b$ lines. If at least one of the following holds:
    \begin{itemize}
    \item[(i)] $n\geq 2$ and $b=n^2+n-1$, or \item[(ii)] $n\geq 4$ and $b=n^2+n-2$,\end{itemize} then $(\mc{S},\mc{A})$ can be completed to an affine plane.
\end{theorem}

A $3$-$(n^d+1, n+1, 1)$-design is an \tit{inversive space}, and in the case that $d=2$ an \tit{inversive plane}. If $\mc{I}$ is a partial $3$-$(n^d+1,n+1,1)$-design, then we can consider the {\em derived} structure at a point $P$, denoted by $\mc{I}_P$: the points of $\mc{I}_P$ are the points of $\mc{I}$, different from $P$, and lines of $\mc{I}_P$ are the circles of $\mc{I}$ containing $P$ with the point $P$ is removed. It is easy to see that $\mc{I}_P$ is a partial $2$-$(n^d,n,1)$- design. 


In Section \ref{section3} we study the completion problem for inversive spaces by considering when the derived structures can be completed to full designs. 
Our most general completion result is given in Theorem \ref{inversive_completion}; this leads to our second main result (Corollary \ref{root_n_completion_of_inversive_plane}) which gives a completion result for inversive planes.


\begin{cor}\label{result_root_n_completion_of_inversive_plane}
    Let $\mc{I}$ be a partial inversive plane, i.e. a partial $3$-$(n^2+1,n+1,1)$-design. Suppose that for every point $P$, at least one of the following holds:
    
    \begin{enumerate}[label=(\roman*)]
        \item there are at least $n^2+n-\sqrt{n}$ circles through $P$ and parallelism is an equivalence relation in $\mc{I}_P$;
        \item there is at most one point in $\mc{I}_P$ of valency less than $n$ and parallelism is an equivalence relation in $\mc{I}_P$;
        \item $\mc{I}_P$ satisfies the requirements of Theorem \ref{DOW_pap_root_n_no_parallel} or Theorem \ref{DOW_pap_with_nothing_2_missing}.
    \end{enumerate}
    Then $\mc{I}$ can be completed to an inversive plane, i.e. a $3$-$(n^2+1,n+1,1)$-design.
\end{cor}

\section{A Completion Result for Affine Spaces}\label{section2}

In this section we first give a general result for when a partial $2$-$(n^d,n,1)$-design can be completed to a full design. We then prove our Main Theorem, improving on Theorem \ref{dowaffine}, yielding a new bound for when a partial affine plane can be completed to an affine plane.

\subsection{A Result for \texorpdfstring{$2$-$(n^d,n,1)$-designs}{2-designs}}

\begin{lemma}\label{valn_acts_as_line}
    Let $\mc{I}$ be a partial $2$-$(n^d,n,1)$-design, $n,d\geq 2$, such that given a line $\ell$ and point $P\notin \ell$, there are at most $\frac{n^d-1}{n-1}-n$ lines through $P$ that do not meet $\ell$. Let $Q$ be a point of $\mc{I}$. If $val(Q)=\frac{n^d-1}{n-1}-1$, then any two points $A,B$ that are unjoined to $Q$ are unjoined to each other.
\end{lemma}
\begin{proof}
    Suppose that $A$ and $B$ are joined by some line $\ell$, then $Q\notin \ell$. Our assumption shows that there are at least $n-1$ lines through $Q$ that meet $\ell$, and hence, all $n-1$ points on $\ell$, different from $A$ (including $B$), are joined to $Q$, a contradiction. 
\end{proof}

\begin{lemma}\label{no_pts_lss_than_n_lines}
    Let $\mc{I} = (\mc{S},\mc{A})$ be a partial $2$-$(n^d,n,1)$-design, $n,d\geq 2$, such that given a line $\ell$ and point $P\notin \ell$, there are at most $\frac{n^d-1}{n-1}-n$ lines through $P$ that do not meet $\ell$. If at most one point, say $Q$, has valency less than $\frac{n^d-1}{n-1}-1$, then $\mc{I}$ can be completed to a full $2$-$(n^d,n,1)$-design.
\end{lemma}
\begin{proof}
    Let $\mc{A'}$ be the collection of sets $\{P\}\cup\{T:T\text{ unjoined to }P\}$ for all points $P$ of valency $\frac{n^d-1}{n-1}-1$. Note that each of these sets consist of $n$ points. Then by Lemma \ref{valn_acts_as_line}, any pair of points contained in any element of $\mc{A'}$ are unjoined. For any pair of unjoined points at least one of the points has valency $\frac{n^d-1}{n-1}-1$, and so the pair lies in an element of $\mc{A'}$.
    
    Let $\mc{C}$ and $\mc{D}$ be two elements of $\mc{A'}$, and suppose that the point $R$ is an element of both. Then $R$ is unjoined to all elements of $\mc{C}$ and $\mc{D}$. Then either $\mc{C}=\mc{D}$ or $val(R)<\frac{n^d-1}{n-1}-1$. So since there is at most one point of valency less than $\frac{n^d-1}{n-1}-1$, any two distinct elements of $\mc{A'}$ share at most one point, namely $Q$.
    
    We conclude that two points are either joined in $\mc{A}$, in which case they lie on exactly one element of $\mc{A}$ and no elements of $\mc{A'}$, or unjoined in $\mc{A}$, in which case they lie on exactly one element of $\mc{A'}$. Therefore $(\mc{S}, \mc{A}\cup\mc{A'})$ is a $2$-$(n^d,n,1)$-design.
\end{proof}

\subsection{A Completion Result for Affine Planes}

We first state a simple result relating certain partial affine planes to partial projective planes as done in \cite{dow1986completion}.

\begin{lemma}\label{affine_to_proj_completion} Let $g_i(n)$, $i=1,\ldots,4$, be as in Theorem \ref{DOW bound} and let $\min(g_1,g_2,g_3,g_4)=g$.
    Let $(\mc{S},\mc{A})$ be a PAP of order $n\geq 2$, with $b$ lines, in which parallelism is an equivalence relation and there are at most $n+1$ parallel classes. If $b>g(n)-1$, then $(\mc{S},\mc{A})$ can be completed to an affine plane. 
\end{lemma}
\begin{proof}
    Construct a new incidence structure from $(\mc{S},\mc{A})$ by adding $n+1$ points to the point set and a line containing these $n+1$ new points to the line set. Let the number of parallel classes be $k\leq n+1$. Pick a one-to-one correspondence between $k$ of the $n+1$ new points and the $k$ parallel classes of $(\mc{S},\mc{A})$, and add the corresponding point to each line of that parallel class. Then this structure is a partial projective plane of order $n$ with $b+1$ lines, no two of which are disjoint, which can be completed to a projective plane by Theorem \ref{DOW bound}. Removing the $n+1$ added points and the added line then gives the desired affine plane.
\end{proof}

\begin{lemma}\label{n+1_pt_gives_n+1_parallel_classes}
    Let $(\mc{S},\mc{A})$ be a PAP of order $n\geq 2$, in which parallelism is an equivalence relation. If there is a point $P$ of valency $n+1$, then $(\mc{S},\mc{A})$ has exactly $n+1$ parallel classes.
\end{lemma}
\begin{proof}
    The $n+1$ lines through $P$ cover $\mc{S}$. Let $\ell$ be a line not through $P$. Then $\ell$ meets $n$ of the lines through $P$ in a single point, and so is parallel to one of the lines through $P$. Therefore every line belongs to the parallel class of one of the lines through $P$, so there are $n+1$ parallel classes.
\end{proof}

\begin{lemma}\label{oneparallel}
    Let $(\mc{S},\mc{A})$ be a PAP. Then parallelism is an equivalence relation on $\mc{A}$ if and only if given a line $\ell$ and a point $P\notin \ell$, there is at most one line through $P$ parallel to $\ell$.
\end{lemma}
\begin{proof}
    Suppose that parallelism is an equivalence relation on $\mc{A}$ and that there are two lines $m_1,m_2$ through $P$ parallel to $\ell$. Then since $m_1$ is parallel to $\ell$, $\ell$ is parallel to $m_2$ and parallelism is an equivalence relation on $\mc{A}$, $m_1$ is parallel to $m_2$. Since $m_1$ and $m_2$ share a point this means they must be equal.

    Conversely, suppose that given a line $\ell$ and point $P\notin \ell$ there is at most one line through $P$ parallel to $\ell$. Parallelism is by definition symmetric and reflexive. Suppose that $\ell_1$ is parallel to $\ell_2$ and $\ell_2$ is parallel to $\ell_3$, and that these three lines are distinct. Then if $\ell_1$ and $\ell_3$ are not parallel,  $\ell_1 \cap \ell_3 = \{Q\}$ and $Q\notin \ell_2$, a contradiction as there cannot be two lines $\ell_1\neq \ell_3$ through $Q$ parallel to $\ell_2$. 
\end{proof}

We first improve the bound $b>n^2$ on the number of lines of Theorem \ref{dowaffine} to $b\geq n^2$.

\begin{theorem}\label{n2-1_completable}
    Let $(\mc{S},\mc{A})$ be a PAP of order $n$ with $b$ lines. If $b\geq n^2$ and parallelism is an equivalence relation on $\mc{A}$, then $(\mc{S},\mc{A})$ can be completed to an affine plane.
\end{theorem}
\begin{proof}
    If $b>n^2$, then the average valency of a point is $\frac{b}{n}>n$, so there must be a point of valency $n+1$. Therefore if $b>n^2$ or a point of valency $n+1$ exists, then $(\mc{S},\mc{A})$ can be completed to an affine plane by Lemma \ref{n+1_pt_gives_n+1_parallel_classes} and Lemma \ref{affine_to_proj_completion}. 
    
    We now consider the remaining case: $b=n^2$ and there is no point of valency $n+1$. Then the average valency of a point is exactly $n$, hence every point has valency $n$. By Lemma \ref{oneparallel}, we have that $(\mc{S},\mc{A})$ satisfies the assumptions of Lemma \ref{no_pts_lss_than_n_lines}, so can be completed to an affine plane.
\end{proof}

To further improve the bound 
 on the number of lines in Theorem \ref{dowaffine}, we will
show the existence of an (incomplete) parallel class and point of valency $n$ outside this parallel class. This will then imply that there are at most $n+1$ parallel classes from which the result follows. We will need the following Lemma.

\begin{lemma}\label{no_full_pt_properties}
    Let $(\mc{S},\mc{A})$ be a PAP of order $n$ with $b=n^2-a$ lines, where $1\leq a \leq n-1$. Suppose that no point has valency $n+1$. Then \begin{itemize}
    \item[(i)] every parallel class has at least $n-a$ lines, and \item[(ii)] there are at most $na$ points of valency less than $n$.
    \end{itemize}
\end{lemma}
\begin{proof}
    \begin{itemize}
    \item[(i)]Suppose to the contrary that the parallel class of some line $\ell$ has size at most $n-a-1$. Then the lines not in this parallel class, of which there are at least $n^2-n+1$, all meet $\ell$, and by the pigeonhole principle some point of $\ell$ has valency $n+1$, a contradiction.

    \item[(ii)]There are $bn = n^3-an$ point-line pairs. Suppose that there are $x$ points of valency less than $n$. Then the number of point-line pairs is
    \begin{align*}n^3-an=\sum_{P\in\mc{S}}val(P) \leq x(n-1) + (n^2-x)n = -x+n^3,\end{align*}
    hence, $x\leq na$. 
    \end{itemize}
\end{proof}

\begin{theorem}\label{pap_n_minus_root_n} Let $g_i(n)$, $i=1,\ldots,4$, be as in Theorem \ref{DOW bound} and let $\min(g_1,g_2,g_3,g_4)=g$.
    Let $(\mc{S},\mc{A})$ be a PAP of order $n$ with $b$ lines. Suppose that parallelism is an equivalence relation on $\mc{A}$, and that $b=n^2-a$, with 
   $1\leq a < \mrm{{min}}(\sqrt{n}, n^2-g(n)+1)$. Then $(\mc{S},\mc{A})$ can be completed to an affine plane.
\end{theorem}
\begin{proof}
    Suppose first that there is a point of valency $n+1$. Then $(\mc{S},\mc{A})$ can be completed to an affine plane by Lemma \ref{n+1_pt_gives_n+1_parallel_classes} and Lemma \ref{affine_to_proj_completion}.

    We may now assume that there is no point of valency $n+1$, so Lemma \ref{no_full_pt_properties} shows that there is a point $P$ of valency $n$. Suppose to the contrary that there are at least $n+2$ parallel classes. Since $P$ has valency $n$, there are two parallel classes, say $\mc{L}_1$ and $\mc{L}_2$ not containing a line through $P$. Denote the set of points covered by the lines of $\mc{A}$ through $P$ by $\mc{M}_P$, then the points covered by the lines of $\mathcal{L}_i$, $i=1,2$ are contained in $\mc{M}_P$. Let $|\mc{L}_i|=n-d_i$, $i=1,2$. It follows from Lemma \ref{no_full_pt_properties} that $d_i\leq a < \sqrt{n}$, and hence $n^2-d_1d_2>n^2-n$. Since every line of $\mathcal{L}_1$ meets every line of $\mathcal{L}_2$, the number of points covered by the lines of $\mathcal{L}_1\cup \mathcal{L}_2$ equals \[(n-d_1)n+(n-d_2)n-(n-d_1)(n-d_2)=n^2-d_1d_2.\] Since these points are all contained in $\mc{M}_P$ and $|\mc{M}_P|=n(n-1),$ it follows that $n^2-d_1d_2\leq n^2-n$, a contradiction. Therefore, there are at most $n+1$ parallel classes, and so by Lemma \ref{affine_to_proj_completion}, $(\mc{S},\mc{A})$ can be completed to an affine plane.

    
    
    
\end{proof}

In the case that $n$ is a square, the following example shows that the above bound is tight.

\begin{example}\label{ex:Baer}
Let $(\mc{S}, \mc{A})$ be a projective plane of order $n$, where $n$ is a square, and consider the point set $\mathcal{B}$ of a  Baer subplane in $(\mc{S}, \mc{A})$ Let $\mc{L}$ be the set of lines of $\mc{A}$ that meet $\mc{B}$ in a single point. The incidence structure $\mathcal{I}=(\mc{S}\setminus\mc{B}, \mc{L})$ has $n^2-\sqrt{n}$ points and lines, all lines have $n$ points, and parallelism is an equivalence relation on $\mc{L}$ with $|\mc{B}| = n+\sqrt{n}+1$ parallel classes. Adjoining $\sqrt{n}$ empty points to $\mathcal{I}$, we obtain a partial affine plane where parallel is an equivalence relation that cannot be embedded in an affine plane (since it has more than $n+1$ parallel classes).
\end{example}

\begin{theorem}\label{pap_n_minus_root_2n} 
 Let $g_i(n)$, $i=1,\ldots,4$, be as in Theorem \ref{DOW bound} and let $\min(g_1,g_2,g_3,g_4)=g$.
    Suppose that $\mc{I} = (\mc{S},\mc{A})$ is a partial affine plane of order $n$ that cannot be extended to an affine plane and suppose that parallelism is an equivalence relation. Then $|\mc{A}| \leq \mrm{max}(g(n)-1, n^2-\sqrt{n})$. Furthermore if $|A| > \mrm{max}(g(n)-1, n^2-f)$, where $f$ is the positive real number satisfying $f(f+1)=2n$, then $\mc{I}$ has a point of valency $0$ and can be embedded in a projective plane of order $n$. Moreover, if $|\mc{A}|=n^2-f$, then either $\mc{I}$ has a point of valency $0$ and can be embedded in a projective plane of order $n$, or there exists an $n$-group divisible design of type $(n-f)^{n+f+2}$ with $n^2-f$ points and lines.
\end{theorem}

\begin{proof}Suppose that $\mc{I} = (\mc{S},\mc{A})$ is a partial affine plane of order $n$ that cannot be extended to an affine plane and suppose that parallelism is an equivalence relation.
    We may assume that $|\mc{A}| > g(n)-1$ and need to show that $|\mc{A}|\leq n^2-f$, where $f$ is the positive real number satisfying $f(f+1)=2n$. Let $f_0:=n^2-|\mc{A}|$.
    
    Let the number of parallel classes be $k$. Then $k\geq n+2$, or otherwise $\mc{I}$ can be extended to an affine plane by Lemma \ref{affine_to_proj_completion}.

    If there is a point of valency $n+1$, then by Lemma \ref{n+1_pt_gives_n+1_parallel_classes} the number of parallel classes is $n+1$, a contradiction. Hence the valency of every point is at most $n$.

    Let $e$ be the number of points of valency $0$. Then since the remaining $n^2-e$ points have at most valency $n$, the number of lines is at most $\frac{n(n^2-e)}{n}$. Therefore $n^2-f_0=|\mc{A}|\leq n^2-e$, so $e\leq f_0$. Hence, since $f_0\leq n+2$, it follows that $e\leq k$.

    Now let $\mc{I'} = (\mc{S'}, \mc{A'})$ be the point-line incidence structure obtained from $\mc{I}$ by adding for each parallel class $\pi$ a new point to the lines of $\pi$. If there are points of valency $0$ in $\mc{I}$, then use these instead of a new point for $e$ of the parallel classes (this can be done since $e\leq k$). We see that $\mc{I'}$ has $|\mc{S}'|=n^2+k-e$ points, $|\mc{A}'|=|\mc{A}|$ lines, each line of $\mc{I'}$ contains $n+1$ points, and every two distinct lines of $\mc{I'}$ meet in a unique point. Furthermore every point of $\mc{I'}$ lies on at least one line.

    For every point $P$ of $\mc{I'}$ let $d_P$ be the number of lines of $\mc{I'}$ containing $P$. Standard double counting arguments show that

    \begin{equation}\label{eqn1}
        \sum_{P\in\mc{S'}}d_P = |\mc{A}'|(n+1)
    \end{equation}

    and
    
    \begin{equation}\label{eqn2}
        \sum_{P\in\mc{S'}}d_P(d_P-1) = |\mc{A}'|(|\mc{A}'|-1).
    \end{equation}

    It follows from the construction of $\mc{I'}$ that $1\leq d_P \leq n$ for all points $P$.

    Next we show that $d_P \geq n-f_0$. Since $d_P\geq 1$, there is a line $l$ through $P$. As every line meets $l$, the total number of lines (which is $n^2-f_0$) is $1+\sum_{X\in l}(d_X-1)$. Using $d_X\leq n$ for the points $X\neq P$ of $l$, it follows that $d_P\geq n-f_0$.

    We have shown that $(n-d_P)(d_P-(n-f_0)) \geq 0$ for all points $P$. It follows that

    \begin{align*}
        0 &\leq \sum_{P\in\mc{S'}}(n-d_P)(d_P-(n-f_0))\\
        &= -\sum_{P\in\mc{S'}}(d_P(d_P-1) + (2n-f_0-1)\sum_{P\in\mc{S'}}d_P - |\mc{S}'|n(n-f_0).
    \end{align*}

    Using equations \eqref{eqn1}, \eqref{eqn2}, and $|\mc{S}'| = n^2+k-e$, it follows that

    \begin{equation*}
        0\leq n((f_0-n)(k-e) + n^2 -f_0n +f_0^2-f_0).
    \end{equation*}

    Dividing by $n$, the resulting inequality is equivalent to

    \begin{equation}
        2n+(n-f_0)((k-e)-(n+2))\leq f_0(f_0+1).\label{eq:end}
    \end{equation}

    Now suppose that $k-e \leq n+1$. Then $\mc{I'}$ has at most $n^2+n+1$ points, so by adding empty points (i.e., points of valency $0$) and using Theorem \ref{DOW bound}, $\mc{I'}$ can be embedded in a projective plane of order $n$, and hence $\mc{I}$ can be embedded in a projective plane of order $n$. Additionally, since $k\geq n+2$, there is an empty point in $\mc{I}$. 

    Now suppose that $k-e\geq n+2$. If $f_0<n$, then Equation \eqref{eq:end} shows that $f_0^2+f_0\geq 2n$. If $f_0\geq n$, then clearly $f_0(f_0+1) > 2n$. Hence $f_0(f_0+1)\geq 2n$, and $f_0\geq f$, where $f$ is the positive real number such that $f(f+1)=2n$.

    Now assume that $k-e\geq n+2$ and $f_0(f_0+1)=2n$. That is, $f_0=f$. Then $f_0< n$ so it follows from Equation \eqref{eq:end} that $k-e=n+2$. Moreover, it folllows that for all $P$,  $(n-d_P)(d_P-(n-f_0))=0$, that is, $d_P\in\{n-f, n\}$ for all points $P$. Hence, if $l$ is a line, then $|\mc{A}'|=n^2-f_0$ implies that one point $P$ of $l$ satisfies $d_P=n-f_0$, whereas $d_X=n$ for all other points $X$ of $l$ (since every line meets $l$ in $\mc{I'}$).

    As $|\mc{A}'|=n^2-f_0=(n+f_0+2)(n-f_0)$, and exactly one point of each line lies on $n-f_0$ lines, it follows that $n+f_0+2$ points of $\mc{I'}$ satisfy $d_P=n-f_0$, and hence $|\mc{S}'|-(n+f+2)=n^2-f$ points of $\mc{I'}$ lie on $n$ lines. Let $\mc{W}$ be the set of these $n^2-f_0$ points. Then every line of $|\mc{A}'|$ has $n$ points in $\mc{W}$, so $\mc{I'}$ induces on $\mc{W}$ an incidence structure $\mc{G}$ with $n^2-f_0$ points, $n^2-f_0$ lines, every point lies on $n$ lines, and every line has $n$ points. Moreover, two lines of $\mc{G}$ have no point in common if and only if these two lines meet in $\mc{I'}$ in a point of $\mc{S'}\setminus \mc{W}$. Hence each point of $\mc{S'}\setminus\mc{W}$ gives rise to a set of $n-f_0$ mutually skew lines of $\mc{G}$. If follows that the structure dual to $\mc{G}$ is an $n$-group divisible design of type $(n-f_0)^{n+f_0+2}$ with $n^2-f_0$ points and lines.
    
\end{proof}

\section{A Completion Result for Inversive Spaces}\label{section3}

In this section we give conditions under which a partial $3$-$(n^d+1,n+1,1)$-design, whose derived structures can be completed, can be completed to a full design. 
Suppose that $\mc{I}_P$ is a partial $2$-$(n^d,n,1)$-design that can be completed to a full $2$-$(n^d,n,1)$-design $\mc{I}'_P$. Then an \tit{added line} is a line in $\mc{I}'_P$ that is not present in $\mc{I}_P$. A point in $\mc{I}'_P$ is a \tit{simple point} if it has exactly one added line through it in $\mc{I}'_P$, that is if it has valency $\frac{n^d-1}{n-1}-1$ in $\mc{I}_P$.


\begin{lemma}\label{simple_point_gives_line}
    Let $\mc{I}$ be a partial $3$-$(n^d+1,n+1,1)$-design with $n,d\geq 2$. Suppose that for every point $P$, $\mc{I}_P$ can be uniquely extended to a full $2$-$(n^d,n,1)$-design $\mc{I}'_P$. 
  Suppose also that there is a point $P$ such that $\mc{I}'_P$ contains an added line $\ell$ containing a simple point $Q$, then
    \begin{itemize}
    \item[(i)] $P$ is a simple point in $\mc{I}'_Q$ and $\{P\} \cup (\ell\setminus \{Q\})$ is an added line in $\mc{I}'_Q$;
    \item[(ii)] for every point $R\in \ell$, every point of $\ell\setminus\{Q,R\}$ is joined to $Q$ by an added line in $\mc{I}'_R$. \end{itemize}
\end{lemma}
\begin{proof}
\begin{itemize}
\item[(i)]
    It is clear that the only triples of points containing $P$ and $Q$ which do not lie on a circle in $\mc{I}$ are those of the form $(P,Q,R)$ where $R\in \ell\setminus \{Q\}$. Hence $P$ is unjoined to only the points of $\ell\setminus \{Q\}$ in $\mc{I}_Q$. Therefore $P$ is a simple point in $\mc{I}'_Q$ and $\{P\} \cup \ell\setminus \{Q\}$ is an added line in $\mc{I}'_Q$.
    \item[(ii)] It follows from (i) that for all $T\in \ell\setminus\{Q,R\}$, the triple of points $(Q, R, T)$ is not contained in a circle of $\mc{I}$. Hence $Q$ and $T$ are joined by an added line in $\mc{I}'_R$.\qedhere
    \end{itemize}
\end{proof}


\begin{theorem}\label{inversive_completion}
    Let $\mc{I}$ be a partial $3$-$(n^d+1,n+1,1)$-design with $n,d\geq 2$. Suppose that for every point $P$, $\mc{I}_P$ can be uniquely extended to a full $2$-$(n^d,n,1)$-design $\mc{I}'_P$, and that in $\mc{I}'_P$ one of the following holds: \begin{itemize}
    \item[(i)] every added line in $\mc{I}'_P$ contains at least $\sqrt{n}$ simple points and through every point of $\mc{I}'_P$ there are at most $\sqrt{n}$ added lines of $\mc{I}'_P$, or \item[(ii)] there is at most one point in $\mc{I}'_P$ with more than one added line of $\mc{I}'_P$ through it. \end{itemize}
    Then $\mc{I}$ can be uniquely completed to a full $3$-$(n^d+1,n+1,1)$-design.
    
\end{theorem}
\begin{proof}
    Let $P$ be a point of $\mc{I}$ and let $\ell$ be an added line in $\mc{I}'_P$. Let $Q$ be a point of $\ell$. Let $\mc{T} = \ell\setminus\{Q\}$. We will show that $\{P\}\cup \mc{T}$ is an added line in $\mc{I}'_Q$. 
    
    First consider the case that $\mc{I}'_Q$ has property (i).
    Let $T\in \mc{T}$, then, since there is no circle through $P,Q,T$ in $\mc{I}$, 
    $P$ and $T$ are unjoined in $\mathcal{I}_Q$. So each of the $n-1$ points of $\mc{T}$ lies on one of the added lines, say $m_1,\ldots,m_s$ through $P$ in $\mathcal{I}'_Q$.
    Since $s\leq \sqrt{n}$, there is at least one of $m_1,\ldots, m_s$, say $m=m_1$, with at  least $\lceil\frac{n-1}{\sqrt{n}}\rceil$ points of $\mc{T}$.

    We now show that every point of $\mc{T}$ is contained in $m$, that is $m=\mc{T}\cup\{P\}$. 
    
    If $Q$ is a simple point in $\mc{I}'_P$ then $m=\mc{T}\cup\{P\}$ by Lemma \ref{simple_point_gives_line}(i), so we may suppose that $Q$ is not a simple point in $\mc{I}'_P$.

    In this case, first suppose that there is some point $R\in \mc{T}$ that is not in $m$, and that $R$ is a simple point in $\mc{I}'_P$. Note that there is no circle through $P,Q,R$ in $\mathcal{I}$, and therefore, $R$ and $P$ are unjoined in $\mathcal{I}_Q$, so in $\mc{I}'_Q$, $R$ is joined to $P$ by an added line. By Lemma \ref{simple_point_gives_line}(ii) $R$ is joined to every point of $\mc{T}$ by an added line in $\mc{I}'_Q$, in particular the points of $\mc{T}$ that lie in $m$. Hence there are at least $\lceil\frac{n-1}{\sqrt{n}}\rceil + 1 \geq \sqrt{n}+1$ added lines through $R$ in $\mc{I}'_Q$, a contradiction.

    We may now assume that every simple point in $\mc{I}'_P$ contained in $\mc{T}$  lies on $m$ in $\mc{I}'_Q$, and that there is some point $R\in \mathcal{T}$ that does not lie in $m$. Then in $\mc{I}'_Q$, $R$ is joined to $P$ by an added line in $\mc{I}'_Q$, and by Lemma \ref{simple_point_gives_line}(ii) $R$ is joined to each of these simple points by an added line. Hence there are at least $\sqrt{n}+1$ added lines through $R$ in $\mc{I}'_Q$, again a contradiction.

    Now consider the case that $\mc{I}'_Q$ has property (ii). Then by Lemma \ref{simple_point_gives_line}(i), $Q$ is a simple point in $\mc{I}'_P$ if and only if $P$ is a simple point in $\mc{I}'_Q$. In this case, $\{P\}\cup \mc{T}$ is an added line in $\mc{I}'_Q$ as required. So we may now assume that $P$ is the  unique point in $\mc{I}'_Q$ through which there is more than one added line. Suppose that $\{P\}\cup \mc{T}$ is not an added line in $\mc{I}'_Q$. Then there is some $A,B\in \mc{T}$ such that $A$ is simple in $\mc{I}'_P$ and $A$ and $B$ lie on different added lines through $P$ in $\mc{I}'_Q$. Then by Lemma \ref{simple_point_gives_line}(ii) there is an added line joining $A$ and $B$ in $\mc{I}'_Q$, hence $A$ has multiple added lines through it in $\mc{I}'_Q$, a contradiction since $A\neq P$.  

    Hence for all points $P$ and all added lines $\ell$ in $\mc{I}'_P$, $\{P\}\cup \ell\setminus\{Q\}$ is an added line in $\mc{I}'_Q$ for all $Q\in \ell$.  Consider now the incidence structure $\mathcal{J}$ with points the points of $\mathcal{I}$, and circles given by the circles of $\mathcal{I}$, together with all circles of the form $\{P\}\cup \ell$ where $\ell$ is an added line in $\mathcal{I}'_P$, for all points $P$ in $\mc{I}$. Call the latter circles {\em added circles}.
    Consider three points $R_1,R_2,R_3$ of $\mathcal{I}$. If $R_1,R_2,R_3$ are contained in a (necessarily unique) circle of $\mathcal{I}$ then no added circle contains $R_1,R_2,R_3$: this is clear for added circles arising from lines in $\mc{I}_{R_i},$ $i=1,2,3$ so suppose that there is a circle $\{R_4\}\cup \ell$ added, where $\ell$ is an added line through $R_1,R_2,R_3$ in $\mathcal{I}'_{R_4}$. We have shown that this implies that in $\mathcal{I}_{R_1}$, $\{R_4\}\cup\ell$ is an added circle, a contradiction since $\ell$ contains $R_2,R_3$ which were joined in $\mathcal{I}_{R_1}.$
    Now assume that $R_1,R_2,R_3$ do no lie on a circle of $\mathcal{I}.$ Then $R_1,R_2,R_3$ lie on a circle $C$ of $\mathcal{J}$ of the form $\{R_1\}\cup \ell$ where $\ell$ is an added line in $\mc{I}_{R_1}$, and we can again see that any added circle of the form $\{R_4\}\cup m$ where $m$ is an added line in $\mathcal{I}_{R_4}$ coincides with $C$ (so $R_4$ is contained in $C$).
    Hence,  $\mc{I}$ can be completed to an inversive plane $\mathcal{J}$.
\end{proof}

\begin{corollary}\label{completion_of_higher_dim_3design}
    Let $\mc{I}$ be a partial $3$-$(n^d+1,n+1,1)$-design with $d\geq 3$, $n\geq 2$. Suppose that in every derived partial $2$-design there is at most one point with valency less than $\frac{n^d-1}{n-1}-1$ and for every line $\ell$ and point $P\notin \ell$ there are at most $\frac{n^d-1}{n-1}-n$ lines through $P$ parallel to $\ell$. Then $\mc{I}$ can be completed to a full $3$-$(n^d+1,n+1,1)$-design.
\end{corollary}
\begin{proof}
    Every derived partial $2$-design satisfies the conditions of Lemma \ref{no_pts_lss_than_n_lines}, and so can be completed to a $2$-design satisfying condition (ii) of Theorem \ref{inversive_completion}.
\end{proof}

\begin{corollary}\label{root_n_completion_of_inversive_plane}
    Let $\mc{I}$ be a partial inversive plane, i.e. a $3$-$(n^2+1,n+1,1)$-design. Suppose that for every point $P$, either
    
    \begin{enumerate}[label=(\roman*)]
        \item There are at least $n^2+n-\sqrt{n}$ circles through $P$ and parallelism is an equivalence relation in $\mc{I}_P$.
        \item There is at most one point in $\mc{I}_P$ of valency less than $n$ and parallelism is an equivalence relation in $\mc{I}_P$.
        \item $\mc{I}_P$ satisfies the requirements of Theorem \ref{DOW_pap_root_n_no_parallel} or Theorem \ref{DOW_pap_with_nothing_2_missing}.
    \end{enumerate}
    Then $\mc{I}$ can be completed to an inversive plane, i.e. a $3-(n^2+1,n+1,1)$-design.
\end{corollary}
\begin{proof}
    The conditions (i) and (ii) both imply that the number of lines in the derived structure is at least $n^2-1$ (and parallel is an equivalence relation), so by Theorem \ref{pap_n_minus_root_n} these derived structures can be completed; in case (iii) this follows from the mentioned theorems. Then every derived structure can be completed to an affine plane which in case (i) satisfies condition (i) of Theorem \ref{inversive_completion}, and in the other cases satisfies condition (ii).
\end{proof}

\bibliographystyle{abbrv}
\bibliography{refcompletion2}

\end{document}